\documentclass[9pt,oneside,english]{amsart}
\usepackage[T1]{fontenc}
\usepackage[latin9]{inputenc}
\usepackage[a4paper]{geometry}
\usepackage{amsthm}
\usepackage{amstext}
\usepackage{amssymb}

\makeatletter
\numberwithin{equation}{section}
\numberwithin{figure}{section}
  \theoremstyle{plain}
  \newtheorem*{question*}{\protect\questionname}
  \theoremstyle{plain}
  \newtheorem*{thm*}{\protect\theoremname}
  \theoremstyle{plain}
  \newtheorem*{cor*}{\protect\corollaryname}
\theoremstyle{plain}
\newtheorem{thm}{\protect\theoremname}[section]
  \theoremstyle{plain}
  \newtheorem{lem}[thm]{\protect\lemmaname}
  \theoremstyle{definition}
  \newtheorem{example}[thm]{\protect\examplename}
  \theoremstyle{plain}
  \newtheorem{prop}[thm]{\protect\propositionname}

\usepackage{xypic}
\theoremstyle{definition}
\newtheorem{parn}{}[subsection]
\newcommand{\bug}{X}

\makeatother

\usepackage{babel}
  \providecommand{\corollaryname}{Corollary}
  \providecommand{\examplename}{Example}
  \providecommand{\lemmaname}{Lemma}
  \providecommand{\propositionname}{Proposition}
  \providecommand{\questionname}{Question}
  \providecommand{\theoremname}{Theorem}
\providecommand{\theoremname}{Theorem}

\begin{document}

\title{Complements of hyperplane sub-bundles in projective space bundles
over $\mathbb{P}^{1}$}

\author{Adrien Dubouloz}

\address{CNRS, Institut de Math\'ematiques de Bourgogne, Universit\'e de
Bourgogne, 9 avenue Alain Savary - BP 47870, 21078 Dijon cedex, France}

\email{adrien.dubouloz@u-bourgogne.fr}

\keywords{Hyperplane sub-bundles, affine-linear bundles, Danilov-Gizatullin
Theorem}

\subjclass[2000]{14L30, 14R05, 14R25 }
\begin{abstract}
We establish that the isomorphy type as an abstract algebraic variety
of the complement of an ample hyperplane sub-bundle $H$ of a $\mathbb{P}^{r-1}$-bundle
$\mathbb{P}(E)\rightarrow\mathbb{P}^{1}$ depends only on the the
$r$-fold self-intersection $(H^{r})\in\mathbb{Z}$ of $H$. In particular
it depends neither on the ambient bundle $\mathbb{P}(E)$ nor on the
choice of a particular ample sub-bundle with given $r$-fold self-intersection.
Our proof exploits the unexpected property that every such complement
comes equipped with the structure of a non trivial torsor under a
vector bundle over the affine line with a double origin. 
\end{abstract}
\maketitle

\section*{Introduction }

The Danilov-Gizatullin Isomorphism Theorem \cite[Theorem 5.8.1]{Gizatullin1977}
(see also \cite{Dubouloz2011a,Flenner2009} for short self-contained
proofs) is a surprising result which asserts that the isomorphy type
as an abstract algebraic variety of the complement of an ample section
$C$ of a $\mathbb{P}^{1}$-bundle $\overline{\nu}:\mathbb{P}(E)\rightarrow\mathbb{P}_{\mathbb{C}}^{1}$
over the complex projective line depends only on the self-intersection
$(C^{2})\geq2$ of $C$. In particular, it depends neither on the
ambient bundle nor on the chosen ample section with fixed self-intersection
$d$. For such a section, the locally trivial fibration $\nu:\mathbb{P}(E)\setminus C\rightarrow\mathbb{P}^{1}$
induced by the restriction of the structure morphism $\overline{\nu}$
is homeomorphic in the euclidean topology to the complex line bundle
$\mathcal{O}_{\mathbb{P}^{1}}\left(-d\right)\rightarrow\mathbb{P}^{1}$.
However the non vanishing of $H^{1}(\mathbb{P}^{1},\mathcal{O}_{\mathbb{P}^{1}}(-d))$
for $d\geq2$ implies that $\nu:\mathbb{P}(E)\setminus C\rightarrow\mathbb{P}^{1}$
is in general a non trivial algebraic $\mathcal{O}_{\mathbb{P}^{1}}(-d)$-torsor,
and the ampleness of $C$ is in fact precisely equivalent to its non
triviality. So the Danilov-Gizatullin Theorem can be rephrased as
the fact that the isomorphy type as an abstract algebraic variety
of the total space of a non trivial $\mathcal{O}_{\mathbb{P}^{1}}(-d)$-torsor
is uniquely determined by its underlying structure of topological
complex line bundle over $\mathbb{P}^{1}$. 

More generally, given a vector bundle $E\rightarrow\mathbb{P}^{1}$
of rank $r\geq3$ and a sub-vector bundle $F\subset E$ of corank
one with quotient line bundle $L$, the complement in the projective
bundle $\overline{\nu}:\mathbb{P}\left(E\right)\rightarrow\mathbb{P}^{1}$
of lines in $E$ of the hyperplane sub-bundle $H=\mathbb{P}\left(F\right)$
inherits the structure of an $F\otimes L^{-1}$-torsor $\nu:\mathbb{P}\left(E\right)\setminus H\rightarrow\mathbb{P}^{1}$.
Similarly as above, the latter is homeomorphic in the euclidean topology
to the complex vector bundle $F\otimes L^{-1}\rightarrow\mathbb{P}^{1}$.
Furthermore, $F\otimes L^{-1}$ is homeomorphic as a complex vector
bundle to $\det(F\otimes L^{-1})\oplus\mathbb{A}_{\mathbb{P}^{1}}^{r-2}\simeq\mathcal{O}_{\mathbb{P}^{1}}(-(H^{r}))\oplus\mathbb{A}_{\mathbb{P}^{1}}^{r-2}$,
where $(H^{r})\in\mathbb{Z}$ denotes the $r$-fold intersection product
of $H$ with itself. So one may ask by analogy with the one dimensional
case if for an ample $H$, the integer $(H^{r})\geq r$ uniquely determines
the isomorphy type of $\mathbb{P}\left(E\right)\setminus H$ as an
algebraic variety. However, since the ampleness of $H$ is in general
no longer equivalent to the non triviality of the torsor $\nu:\mathbb{P}\left(E\right)\setminus H\rightarrow\mathbb{P}^{1}$,
the following problem seems more natural: 
\begin{question*}
Is the isomorphy type as an abstract algebraic variety of the total
space of a nontrivial torsor under an algebraic vector bundle $G\rightarrow\mathbb{P}^{1}$
uniquely determined by the isomorphy type of $G$ as a topological
complex vector bundle over $\mathbb{P}^{1}$, that is, by the rank
and the degree of $G$ ? 
\end{question*}
While our results imply in particular that non trivial torsors under
homeomorphic complex vector bundles do indeed have isomorphic total
spaces, the answer to this question is negative in general: there
exists torsors under non homeomorphic vector bundles or rank $r\geq2$
which have isomorphic total spaces as algebraic varieties. For instance,
the complement of the diagonal in $\mathbb{P}^{1}\times\mathbb{P}^{1}$
is an affine surface $S$ which inherits two structures  of non trivial
$\mathcal{O}_{\mathbb{P}^{1}}(-2)$-torsor via the first and the second
projections ${\rm pr}_{i}:S\rightarrow\mathbb{P}^{1}$, $i=1,2$.
The Picard group of $S$ is isomorphic to $\mathbb{Z}$ and for every
$k\in\mathbb{Z}$ the line bundles ${\rm pr}_{1}^{*}\mathcal{O}_{\mathbb{P}^{1}}(k)$
and ${\rm pr}_{2}^{*}\mathcal{O}_{\mathbb{P}^{1}}(-k)$ are isomorphic.
It follows that for every $k\in\mathbb{Z}$, the total space of ${\rm pr}_{1}^{*}\mathcal{O}_{\mathbb{P}^{1}}(k)$
is simultaneously the total space of an $\mathcal{O}_{\mathbb{P}^{1}}(-2)\oplus\mathcal{O}_{\mathbb{P}^{1}}(k)$-torsor
and of an $\mathcal{O}_{\mathbb{P}^{1}}(-2)\oplus\mathcal{O}_{\mathbb{P}^{1}}(-k)$-torsor
over $\mathbb{P}^{1}$ via the first and the second projection respectively.
In particular, for every $k\neq0$, we obtain an affine variety which
is simultaneously the total space of non trivial torsors under complex
vector bundles over $\mathbb{P}^{1}$ with different topological types. 

So the degree of the complex vector bundle $G\rightarrow\mathbb{P}^{1}$
is not the appropriate numerical invariant to classify isomorphy types
of total spaces of non trivial algebraic $G$-torsors. In contrast,
our main result can be summarized as follows: 
\begin{thm*}
\label{thm:Main_P1_Thm} The total space of a torsor $\nu:V\rightarrow\mathbb{P}^{1}$
under a vector bundle $G\rightarrow\mathbb{P}^{1}$ is an affine variety
if and only if $\nu:V\rightarrow\mathbb{P}^{1}$ is a non trivial
torsor. If so, the isomorphy type of $V$ as an abstract algebraic
variety is uniquely determined by the rank of $G$ and the absolute
value of $\deg G+2$. 
\end{thm*}
The role of the integer $\left|\deg G+2\right|$ may look surprising
but the latter is intimately related to the subgroup of the Picard
group ${\rm Pic}\left(V\right)\simeq\mathbb{Z}$ of $V$ generated
by the canonical bundle $K_{V}=\det(\Omega_{V}^{1})$ of $V$. Indeed,
for a $G$-torsor $\nu:V\rightarrow\mathbb{P}^{1}$, the relative
cotangent bundle $\Omega_{V/\mathbb{P}^{1}}^{1}$ is isomorphic to
$\nu^{*}G^{\vee}$ and so it follows from the relative cotangent exact
sequence 
\[
0\rightarrow\nu^{*}\Omega_{\mathbb{P}^{1}}^{1}\rightarrow\Omega_{V}^{1}\rightarrow\Omega_{V/\mathbb{P}^{1}}^{1}\simeq\nu^{*}G^{\vee}\rightarrow0
\]
that $K_{V}\simeq\nu^{*}(\det G^{\vee}\otimes\Omega_{\mathbb{P}^{1}}^{1})$
whence that the subgroup of ${\rm Pic}(V)$ generated by $K_{V}$
is isomorphic to $\left|\deg G+2\right|\mathbb{Z}\subset\mathbb{Z}$. 

In the particular case where $\nu:V\rightarrow\mathbb{P}^{1}$ is
a torsor arising as the complement $\mathbb{P}\left(E\right)\setminus H$
of an ample hyperplane sub-bundle $H$, one has $\deg G=-\left(H^{r}\right)\leq-r$
and so the above characterization specializes to the following generalization
of the geometric form of the Danilov-Gizatullin Theorem: 
\begin{cor*}
The isomorphy type as an abstract algebraic variety of the complement
of an ample hyperplane sub-bundle $H$ of a $\mathbb{P}^{r-1}$-bundle
$\rho:\mathbb{P}\left(E\right)\rightarrow\mathbb{P}^{1}$ depends
only on the $r$-fold self-intersection $\left(H^{r}\right)\geq r$
of $H$. 
\end{cor*}
The proof of the above results exploits a hidden and unexpected feature
of total spaces of non trivial torsors $\nu:V\rightarrow\mathbb{P}^{1}$
under a vector bundle $G\rightarrow\mathbb{P}^{1}$ of rank $r\geq1$,
namely the existence on every such $V$ of the structure $\rho:V\rightarrow\bug$
of a non trivial torsor under a vector bundle $\tilde{G}\rightarrow\bug$
of the same rank $r$, on a non separated scheme $\bug$, isomorphic
to the affine line with a double origin. The structure of these bundles
$\tilde{G}\rightarrow\bug$ is very similar to that of vector bundles
on $\mathbb{P}^{1}$: in particular the existence of a covering of
$\bug$ by two open subsets isomorphic to $\mathbb{C}$ and intersecting
along $\mathbb{C}^{*}$ implies that as a topological complex vector
bundle, $\tilde{G}\rightarrow\bug$ is uniquely determined by a homotopy
class of maps $S^{1}\rightarrow{\rm GL}_{r}\left(\mathbb{C}\right)$,
whence simply by the {}``degree'' of its determinant. In this setting,
we establish that the total space of a non trivial torsor $\nu:V\rightarrow\mathbb{P}^{1}$
under a vector bundle $G\rightarrow\mathbb{P}^{1}$ of degree $d$
carries the structure of a non trivial torsor $\rho:V\rightarrow\bug$
under a vector bundle $\tilde{G}\rightarrow\bug$ of degree $d+2$
uniquely determined by $G$. While there exists infinite moduli for
isomorphy type of total spaces of non trivial torsors under a line
bundle on $\bug$, the Danilov-Gizatullin Theorem can be re-interpreted
as the fact that the total spaces of non trivial torsors $\nu:V\rightarrow\mathbb{P}^{1}$
under $G=\mathcal{O}_{\mathbb{P}^{1}}(-d)$ are all isomorphic as
torsors $\rho:V\rightarrow\bug$ under the corresponding line bundle
$\tilde{G}$. For higher dimensional vector bundles $\tilde{G}\rightarrow\bug$,
we establish in contrast that the isomorphy type as an abstract variety
of the total space of a non trivial $\tilde{G}$-torsor $\rho:V\rightarrow\bug$
is uniquely determined by the absolute value of the degree of $\tilde{G}$.
\\

The article is organized as follows: the first section recalls basic
properties of torsors under vector bundles. Then section two is devoted
to the study of isomorphy types of total spaces of non trivial torsors
under vector bundles on the affine line with a double origin. These
results are applied in the third section to the case of torsors under
vector bundles on $\mathbb{P}^{1}$.

\section{Recollection on affine-linear bundles }

Here, to fix the notation and convention that will be used in the
sequel, we briefly recall classical facts about vector bundles, projective
bundles and affine-linear bundles.

\subsection{Vector bundles and projective bundles }

\begin{parn} A \emph{vector bundle} \emph{of rank} $r\geq1$ on a
scheme $X$ is the relative spectrum $p:F={\rm Spec}_{X}({\rm Sym}(\mathcal{F}^{\vee}))\rightarrow X$
of the symmetric algebra of the dual of a locally free coherent $\mathcal{O}_{X}$-module
$\mathcal{F}$ of rank $r$. The $X$-scheme $F$ represents the contravariant
functor from the category of schemes over $X$ to the category of
abelian groups which associates to an $X$-scheme $f:Y\rightarrow X$
the group $\Gamma\left(Y,f^{*}\mathcal{F}\right)$ of global sections
of $f^{*}\mathcal{F}$ over $Y$. It follows in particular that $p:F\rightarrow X$
is a locally constant commutative affine group scheme over $X$, with
group law $F\times_{X}F\rightarrow F$ induced by the diagonal homomorphism
of $\mathcal{O}_{X}$-module $\mathcal{F}^{\vee}\rightarrow\mathcal{F}^{\vee}\oplus\mathcal{F}^{\vee}$. 

\end{parn}

\begin{parn} Given a vector bundle $p:E={\rm Spec}_{X}({\rm Sym}(\mathcal{E}^{\vee}))\rightarrow X$,
the \emph{projective bundle of lines in} $E$ is the relative proj
$\overline{\nu}:\mathbb{P}\left(E\right)={\rm Pro}j_{X}({\rm Sym}(\mathcal{E}^{\vee}))\rightarrow X$
of the symmetric algebra of $\mathcal{E}^{\vee}$ considered as a
graded quasi-quoherent $\mathcal{O}_{X}$-algebra ${\rm Sym}(\mathcal{E}^{\vee})=\bigoplus_{n\geq0}{\rm Sym}^{n}(\mathcal{E}^{\vee})$.
There is a canonical surjection $\overline{\nu}^{*}E^{\vee}\rightarrow\mathcal{O}_{\mathbb{P}\left(E\right)}(1)$
on $\mathbb{P}\left(E\right)$ which yields dually a closed embedding
$\mathcal{O}_{\mathbb{P}\left(E\right)}(-1)\hookrightarrow\overline{\nu}^{*}E$
of the tautogical line bundle as a line sub-bundle of $\overline{\nu}^{*}E$.
More generally, if $f:Y\rightarrow X$ is a scheme over $X$ then
an $X$-morphism $\tilde{f}:Y\rightarrow\mathbb{P}\left(E\right)$
is uniquely determined by a sub-line bundle of $f^{*}E$, which then
coincides with $f^{*}\mathcal{O}_{\mathbb{P}\left(E\right)}(-1)$.
In other word, $\mathbb{P}\left(E\right)$ represents the functor
which associates to an $X$-scheme $f:Y\rightarrow X$ the set of
sub-line bundles of $f^{*}E$. Recall that if $L={\rm Spec}_{X}\left({\rm Sym}(\mathcal{L}^{\vee})\right)\rightarrow X$
is a line bundle on $X$, then the canonical isomorphism ${\rm Sym}\left((\mathcal{E}\otimes\mathcal{L})^{\vee}\right)\simeq\bigoplus_{n\geq0}{\rm Sym^{n}(\mathcal{E}^{\vee})\otimes}(\mathcal{L}^{\vee})^{\otimes n}$
of graded $\mathcal{O}_{X}$-algebras yields an isomorphism $\varphi:\mathbb{P}\left(E\right)\stackrel{\sim}{\rightarrow}\mathbb{P}\left(E\otimes L\right)$
of schemes over $X$ with $\varphi^{*}\mathcal{O}_{\mathbb{P}\left(E\otimes L\right)}\left(-1\right)\simeq\mathcal{O}_{\mathbb{P}\left(E\right)}\left(-1\right)\otimes\overline{\nu}^{*}L$. 

\end{parn}

\subsection{Torsors under vector bundles}

\begin{parn} Torsors under a vector bundle on a scheme are the analogues
in a relative setting of affine spaces attached to a vector space.
Namely, given a vector bundle $p:F={\rm Spec}_{X}({\rm Sym}(\mathcal{F}^{\vee}))\rightarrow X$,
a \emph{principal homogeneous} $F$-\emph{bundle}, or an $F$-\emph{torsor},
is a scheme $\nu:V\rightarrow X$ equipped with an action $\mu:F\times_{X}V\rightarrow V$
of the group scheme $F$ for which there exists a covering of $X$
by open subsets $\{U_{i}\}_{i\in I}$ such that for every $i\in I$,
$V\mid_{U_{i}}=\nu^{-1}(U_{i})$ is equivariantly isomorphic to $F\mid_{U_{i}}$
acting on itself by translations. Given a collection of equivariant
trivializations $\tau_{i}:V\mid_{U_{i}}\stackrel{\sim}{\rightarrow}F\mid_{U_{i}}$,
$i\in I$, it follows that for every $i,j\in I$, $\tau_{i}\circ\tau_{j}^{-1}\mid_{U_{i}\cap U_{j}}$
is an equivariant automorphism of $F\mid_{U_{i}\cap U_{j}}$ whence
is a translation determined by a section $g_{ij}\in\Gamma(U_{i}\cap U_{j},F)=\Gamma(U_{i}\cap U_{j},\mathcal{F})$
of $F\mid_{U_{i}\cap U_{j}}$. Clearly, $g_{ik}\mid_{U_{i}\cap U_{j}\cap U_{k}}=g_{ij}\mid_{U_{i}\cap U_{j}\cap U_{k}}+g_{jk}\mid_{U_{i}\cap U_{j}\cap U_{k}}$
for every $i,j,k\in I$, that is, $(g_{ij})_{i,j\in I}$ is a \v{C}ech
$1$-cocycle with value in the sheaf $\mathcal{F}$ for the open covering
$\left\{ U_{i}\right\} _{i\in I}$. Changing the trivializations $\tau_{i}$
replaces $(g_{ij})_{i,j\in I}$ by a cohomologous cocycle and the
cohomology class is unaltered if $V$ is replaced by an isomorphic
torsor. Thus $\nu:V\rightarrow X$ defines a class $c(V)\in\check{H}^{1}(\{U_{i}\}_{i\in I},F)=\check{H}^{1}(\{U_{i}\}_{i\in I},\mathcal{F})$
and a standard argument eventually shows that there is a one-to-one
correspondence between isomorphy classes of $F$-torsors and elements
of the cohomology group $\check{H}^{1}(X,F)\simeq H^{1}(X,F)$, with
$0\in H^{1}(X,F)$ corresponding the the trivial $F$-torsor $p:F\rightarrow X$
(see e.g. \cite[16.4.9]{Grothendieck}). This implies in particular
that every $F$-torsor on an affine scheme $X$ is isomorphic to the
trivial one. 

\end{parn}

\begin{parn} Recall that the relative cotangent bundle $\Omega_{F/X}^{1}$
of a vector bundle $p:F\rightarrow X$ is canonicaly isomorphic to
$p^{*}F^{\vee}$ \cite[16.5.15]{Grothendieck}. One checks using the
above local description that this holds more generally for any $F$-torsor
$\nu:V\rightarrow X$, providing a canonical short exact sequence
$0\rightarrow\nu^{*}\Omega_{X}^{1}\rightarrow\Omega_{V}^{1}\rightarrow\Omega_{V/X}^{1}\simeq\nu^{*}F^{\vee}\rightarrow0$
of vector bundles on $V$. If $X$ is normal then the natural homomorphism
$\nu^{*}:{\rm Pic}\left(X\right)\stackrel{\sim}{\rightarrow}{\rm Pic}\left(V\right)$
is an isomorphism; the relative canonical bundle $K_{V/X}=\det(\Omega_{V/X}^{1})$
and the canonical bundle $K_{V}=\det(\Omega_{V}^{1})$ of $V$ then
coincide respectively with the images by $\nu^{*}$ of the line bundles
$\det(F^{\vee})$ and $K_{X}\otimes\det(F^{\vee})$ on $X$. 

\end{parn}

\begin{parn} Given a vector bundle $p:F\rightarrow X$, every class
$c\in H^{1}(X,F)$ coincides via the canonical isomorphism $H^{1}(X,F)\simeq{\rm Ext}^{1}\left(\mathbb{A}_{X}^{1},F\right)$
with the isomorphy class of an extension $0\rightarrow F\rightarrow E\rightarrow\mathbb{A}_{X}^{1}\rightarrow0$
of vector bundles on $X$, where $\mathbb{A}_{X}^{1}=X\times\mathbb{A}^{1}$
denotes the trivial line bundle on $X$. The inclusion $F\hookrightarrow E$
induces a closed immersion of $\mathbb{P}\left(F\right)$ into $\overline{\nu}:\mathbb{P}\left(E\right)\rightarrow X$
as the zero locus of the regular section of $\mathcal{O}_{\mathbb{P}\left(E\right)}(1)$
deduced from the composition $\mathcal{O}_{\mathbb{P}\left(E\right)}(-1)\hookrightarrow\overline{\nu}^{*}E\rightarrow\overline{\nu}^{*}\mathbb{A}_{X}^{1}$,
and the complement $\mathbb{P}\left(E\right)\setminus\mathbb{P}\left(F\right)$
is then isomorphic as a scheme over $X$ to the total space of an
$F$-torsor $\nu:V\rightarrow X$ with isomorphy class $c$. In particular,
the trivial extension $E=F\oplus\mathbb{A}_{X}^{1}$ corresponds to
the canonical open immersion of $F$ into $\mathbb{P}(F\oplus\mathbb{A}_{X}^{1})$.
More generally, given a line bundle $L\rightarrow X$ on $X$ and
a short exact sequence of vector bundles $0\rightarrow F\rightarrow E\rightarrow L\rightarrow0$,
the complement of $\mathbb{P}\left(F\right)\simeq\mathbb{P}(F\otimes L^{-1})$
in $\mathbb{P}\left(E\right)\simeq\mathbb{P}(E\otimes L^{-1})$ inherits
the structure of an $F\otimes L^{-1}$-torsor with isomorphy class
in $H^{1}(X,F\otimes L^{-1})\simeq{\rm Ext}^{1}\left(L,F\right)$
given by the class of the extension $0\rightarrow F\rightarrow E\rightarrow L\rightarrow0$. 

\end{parn}

\subsection{Affine-linear bundles }

\begin{parn} Affine-linear bundles over a scheme $X$ from a sub-class
of the class of locally trivial $\mathbb{A}^{n}$-bundle over $X$,
namely, an \emph{affine-linear bundle} of rank $r\geq1$ over $X$
is an $X$-scheme $\nu:V\rightarrow X$ for which there exists a open
covering $\{U_{i}\}_{i\in I}$ of $X$ and a collection of isomorphisms
$\tau_{i}:V\mid_{U_{i}}\stackrel{\sim}{\rightarrow}\mathbb{A}_{U_{i}}^{r}$
of schemes over $U_{i}$ such that for every $i,j\in I$, $\tau_{ij}=\tau_{i}\circ\tau_{j}^{-1}\mid_{U_{i}\cap U_{j}}$
is an affine automorphism of $\mathbb{A}_{U_{i}\cap U_{j}}^{r}={\rm Spec}_{U_{i}\cap U_{j}}(\mathcal{O}_{U_{i}\cap U_{j}}[x_{1},\ldots,x_{r}])$.
This means that there exists $(A_{ij},T_{ij})\in{\rm Aff}_{r}(U_{i}\cap U_{j})={\rm GL}_{r}(U_{i}\cap U_{j})\rtimes\mathbb{G}_{a}^{r}(U_{i}\cap U_{j})$
such that $\tau_{ij}(x_{1},\ldots,x_{r})=A_{ij}(x_{1},\ldots,x_{r})+T_{ij}$
for every $i,j\in I$. It follows that isomorphy classes of affine-linear
bundles of rank $r$ are in one-to-one correspondence with that of
principal homogeneous bundles under the affine group ${\rm Aff}_{r}={\rm GL}_{r}\rtimes\mathbb{G}_{a}^{r}$. 

\end{parn}

\begin{parn} Of course, every torsor under a vector bundle of rank
$r\geq1$ is an affine linear bundle of rank $r$. Conversely, let
$\nu:V\rightarrow X$ be an affine-linear bundle of rank $r\geq1$
with trivializations $\tau_{i}:V\mid_{U_{i}}\stackrel{\sim}{\rightarrow}\mathbb{A}_{U_{i}}^{r}$.
Then for every triple of indices $i,j,k\in I$, the identities 
\[
\begin{cases}
A_{ik} & =A_{jk}A_{ij}\\
^{t}T_{ik} & =A_{jk}\cdot^{t}T_{ij}+^{t}T_{jk}
\end{cases}
\]
hold in ${\rm GL}_{r}(U_{i}\cap U_{j}\cap U_{k})$ and $\mathbb{G}_{a}^{r}(U_{i}\cap U_{j}\cap U_{k})$
respectively. If we identify ${\rm Aff}_{r}(U_{i}\cap U_{j})$ with
the sub-group of ${\rm GL}_{r+1}(U_{i}\cap U_{j})$ consisting of
matrices of the form ${\displaystyle \tilde{A}_{ij}=\left(\begin{array}{cc}
A_{ij} & ^{t}T_{ij}\\
0 & 1
\end{array}\right)}$ these relations say equivalently that $(\tilde{A}_{ij})_{i,j\in I}$
and $(A_{ij})_{i,j\in I}$ are \v{C}ech cocycles with value in ${\rm GL}_{r+1}$
and ${\rm GL}_{r}$ respectively for the open covering $\{U_{i}\}_{i\in I}$
of $X$. These define respectively a vector bundle $E\rightarrow X$
of rank $r+1$ and a sub-vector bundle $F$ of $E$ fitting in a short
exact sequence $0\rightarrow F\rightarrow E\rightarrow\mathbb{A}_{X}^{1}\rightarrow0$
of vector bundles on $X$. By construction, $\nu:V\rightarrow X$
is isomorphic as a scheme over $X$ to the complement of $\mathbb{P}\left(F\right)$
in $\overline{\nu}:\mathbb{P}\left(E\right)\rightarrow X$, whence
can be equipped with the structure of an $F$-torsor. Changing the
trivializations $\tau_{i}$ by means of affine automorphisms changes
the cocycle $(A_{ij})_{i,j\in I}$ for a cohomologous one and the
cohomology class in $H^{1}(X,{\rm GL}_{r})$ is unaltered if we replace
$\nu:V\rightarrow X$ by an isomorphic affine-linear bundle. Therefore,
the vector bundle $F$ for which an affine-linear bundle can be equipped
with the structure of an $F$-torsor is uniquely determined up to
isomorphism. Similarly, the class in ${\rm Ext}^{1}(\mathbb{A}_{X}^{1},F)$
defined by $(\tilde{A}_{ij})_{i,j\in I}$ is unaltered if we change
the $\tau_{i}$'s or replace $\nu:V\rightarrow X$ by an isomorphic
affine-linear bundle, and so the isomorphy class in $H^{1}(X,F)\simeq{\rm Ext}(\mathbb{A}_{X}^{1},F)$
of $\nu:V\rightarrow X$ as an $F$-torsor is also uniquely determined. 

\end{parn}

\section{Affine-linear bundles over the affine with a double origin }

The affine line with a double origin is the scheme $\delta:\bug\rightarrow\mathbb{A}^{1}={\rm Spec}(\mathbb{C}\left[x\right])$
obtained by gluing two copies $X_{\pm}$ of the affine line $\mathbb{A}^{1}={\rm Spec}(\mathbb{C}\left[x\right])$,
with respective origins $o_{\pm}$, by the identity along the open
subsets $X_{\pm}^{*}=X_{\pm}\setminus\left\{ o_{\pm}\right\} $. It
comes equipped with a canonical covering $\mathcal{U}$ by the open
subsets $X_{+}$ and $X_{-}$. The morphism $\delta$ is induced by
the identity morphism on $\bug_{\pm}$ and restricts to an isomorphism
$\bug\setminus\left\{ o_{\pm}\right\} \simeq{\rm Spec}\left(\mathbb{C}[x^{\pm1}]\right)$. 

Since every automorphism of $\bug$ is induced by an automorphism
of $\bug_{+}\sqcup\bug_{-}$ of the form $\bug_{\pm}\ni x\mapsto ax\in\bug_{\varepsilon\cdot\pm}$,
where $a\in\mathbb{C}^{*}$ and $\varepsilon=\pm1$, the automorphism
group ${\rm Aut}\left(\bug\right)$ of $\bug$ is isomorphic to $\mathbb{G}_{m}\times\mathbb{Z}_{2}$.
In what follows, we denote by $\theta=\left(1,-1\right)\in\mathbb{G}_{m}\times\mathbb{Z}_{2}$
the automorphism which exchanges the open subsets $\bug_{\pm}$ of
$\bug$.

\subsection{Vector bundles on the affine line with a double origin. }

\begin{parn} \label{par:VectB-on-bug} Since every line bundle on
$\bug$ becomes trivial on the canonical covering $\mathcal{U}$,
the Picard group ${\rm Pic}(\bug)$ of $\bug$ is isomorphic to $\check{H}^{1}(\mathcal{U},\mathcal{O}_{\bug}^{*})\simeq\mathbb{C}[x^{\pm1}]^{*}/\mathbb{C}^{*}\simeq\mathbb{Z}$.
In what follows we fix as a generator for ${\rm Pic}\left(\bug\right)$
the class of the line bundle $p:\mathbb{L}\rightarrow\bug$ with trivializations
$\mathbb{L}\mid_{\bug_{\pm}}\simeq{\rm Spec}\left(\mathbb{C}\left[x\right][u_{\pm}]\right)$
and transition isomorphism $\tau_{\pm}:\bug_{+}^{*}\times\mathbb{A}^{1}\stackrel{\sim}{\longrightarrow}\bug_{-}^{*}\times\mathbb{A}^{1}$,
$(x,u_{+})\mapsto(x,xu_{+})$. The pull-back of $\mathbb{L}$ by the
automorphism $\theta$ of $\bug$ which exchanges the two open subsets
$\bug_{\pm}$ of $\bug$ is isomorphic to the dual $\mathbb{L}^{-1}$
of $\mathbb{L}$. A line bundle $L\rightarrow\bug$ isomorphic to
$\mathbb{L}^{k}$ for some $k\in\mathbb{Z}$ is said to be of \emph{degree}
$-k$. 

More generally, every vector bundle $E\rightarrow\bug$ of rank $r\geq2$
becomes trivial on the canonical covering $\mathcal{U}$ of $\bug$,
whence is determined up to isomorphism by the equivalence class of
a matrix $M\in{\rm GL}_{r}\left(\mathbb{C}[x^{\pm1}]\right)$ in the
double quotient $\check{H}^{1}(\mathcal{U},{\rm GL}_{r})\simeq{\rm GL}_{r}\left(\mathbb{C}\left[x\right]\right)\backslash{\rm GL}_{r}(\mathbb{C}[x^{\pm1}])/{\rm GL}_{r}\left(\mathbb{C}\left[x\right]\right)$.
Since for a suitable $n\geq0$, $E\otimes\mathbb{L}^{n}$ is determined
by a matrix $M\in\mathcal{M}_{r}\left(\mathbb{C}\left[x\right]\right)\cap{\rm GL}_{r}\left(\mathbb{C}[x^{\pm1}]\right)$
equivalent in the double quotient ${\rm GL}_{r}\left(\mathbb{C}\left[x\right]\right)\backslash\mathcal{M}_{r}\left(\mathbb{C}\left[x\right]\right)/{\rm GL}_{r}\left(\mathbb{C}\left[x\right]\right)$
to its Smith diagonal normal form, it follows that $E$ splits into
a direct sum of line bundles, i.e., is isomorphic to $\bigoplus_{i=1}^{r}\mathbb{L}^{k_{i}}$
for suitable $k_{1},\ldots,k_{r}\in\mathbb{Z}$. The Grothendieck
group $K_{0}\left(\bug\right)$ of vector bundles on $\bug$ is described
as follows: 

\end{parn}
\begin{lem}
\label{lem:K0-bug} The map $K_{0}\left(\bug\right)\rightarrow{\rm Pic}\left(\bug\right)\oplus\mathbb{Z}$,
$\left[E\right]\mapsto(\det E,{\rm rk}\left(E\right))$ is an isomorphism
of groups. \end{lem}
\begin{proof}
The only non trivial part is to show that this map is injective, or
equivalently, that if $E\rightarrow\bug$ is a vector bundle of rank
$r\geq2$ then $E$ and $\det\left(E\right)\oplus\mathbb{A}_{\bug}^{r-1}$
have the same class in $K_{0}\left(\bug\right)$. We proceed by induction
on $r\geq2$. Since every vector bundle on $\bug$ is decomposable,
we may assume that $E=\mathbb{L}^{m}\oplus E'$ where $k\in\mathbb{Z}$
and $E'$ is a vector bundle of rank $r-1$. By induction hypothesis,
$E'$ has the same class in $K_{0}\left(\bug\right)$ as $\det\left(E'\right)=\mathbb{L}^{n}\oplus\mathbb{A}_{\bug}^{r-2}$
where $n=\deg\left(E'\right)$, and so it is enough to check that
for every $m,n\in\mathbb{Z}$, $\mathbb{L}^{m}\oplus\mathbb{L}^{n}$
and $\mathbb{L}^{m+n}\oplus\mathbb{A}_{\bug}^{1}$ have the same class
in $K_{0}\left(\bug\right)$. If either $m$ or $n$ is equal to zero
then we are done. Otherwise, up to changing $\mathbb{L}^{m}\oplus\mathbb{L}^{n}$
for its dual and exchanging the roles of $m$ and $n$, we may assume
that either $0<m\leq n$ or $m<0<n$. In the first case, the matrix 

\[
M=\left(\begin{array}{cc}
x^{m} & 1\\
0 & x^{n}
\end{array}\right)=\left(\begin{array}{cc}
0 & 1\\
-1 & x^{n-1}
\end{array}\right)\left(\begin{array}{cc}
x^{m+n} & 0\\
0 & 1
\end{array}\right)\left(\begin{array}{cc}
1 & 0\\
x^{m-1} & 1
\end{array}\right)\in{\rm GL}_{2}\left(\mathbb{C}\left[x^{\pm1}\right]\right)
\]
is equivalent in $\check{H}^{1}(\mathcal{U},{\rm GL}_{2})$ to ${\rm diag}\left(x^{m+n},1\right)$
and defines an extension $0\rightarrow\mathbb{L}^{m}\rightarrow\mathbb{L}^{m+n}\oplus\mathbb{A}_{\bug}^{1}\rightarrow\mathbb{L}^{n}\rightarrow0$.
Hence $\mathbb{L}^{m}\oplus\mathbb{L}^{n}$ and $\mathbb{L}^{m+n}\oplus\mathbb{A}_{\bug}^{1}$
have the same class in $K_{0}\left(\bug\right)$. The second case
follows from the same argument using the fact that the matrix 
\[
N=\left(\begin{array}{cc}
1 & x^{m}\\
0 & x^{m+n}
\end{array}\right)=\left(\begin{array}{cc}
1 & 0\\
x^{n} & 1
\end{array}\right)\left(\begin{array}{cc}
x^{m} & 0\\
0 & x^{n}
\end{array}\right)\left(\begin{array}{cc}
x^{-m} & 1\\
-1 & 0
\end{array}\right)\in{\rm GL}_{2}\left(\mathbb{C}\left[x^{\pm1}\right]\right)
\]
is equivalent in $\check{H}^{1}(\mathcal{U},{\rm GL}_{2})$ to ${\rm diag}\left(x^{m},x^{n}\right)$
and defines an extension $0\rightarrow\mathbb{A}_{\bug}^{1}\rightarrow\mathbb{L}^{m}\oplus\mathbb{L}^{n}\rightarrow\mathbb{L}^{m+n}\rightarrow0$. 
\end{proof}

\subsection{Affine-linear bundles of rank one}

\indent\newline\noindent Here we review the classification of affine-linear
bundle of rank one over $\bug$ following \cite{DubG03,Fieseler1994}. 

\begin{parn} \label{Par:Rank-one-bug} In view of the above description
of ${\rm Pic}(\bug)$,  every affine-linear bundle $\rho:S\rightarrow\bug$
of rank one over $\bug$ is an $\mathbb{L}^{k}$-torsor for a certain
$k\in\mathbb{Z}$. We deduce from the isomorphism $H^{1}(\bug,\mathbb{L}^{k})\simeq\check{H}^{1}(\mathcal{U},\mathbb{L}^{k})\simeq\mathbb{C}\left[x^{\pm1}\right]/\langle x^{k}\mathbb{C}\left[x\right]+\mathbb{C}\left[x\right]\rangle$
that every nontrivial $\mathbb{L}^{k}$-torsor $\rho:S\rightarrow\bug$
is isomorphic to a one obtain by gluing $\bug_{+}\times\mathbb{A}^{1}$
and $\bug_{-}\times\mathbb{A}^{1}$ over $\bug_{+}\cap\bug_{-}$ by
an isomorphism of the form $(x,u_{+})\mapsto(x,x^{k}u_{+}+g\left(x\right))$
for a Laurent polynomial $g\left(x\right)\in\mathbb{C}\left[x^{\pm1}\right]$
with non zero residue class in $\mathbb{C}\left[x^{\pm1}\right]/\langle x^{k}\mathbb{C}\left[x\right]+\mathbb{C}\left[x\right]\rangle$.
This implies in turn that the total space of a nontrivial $\mathbb{L}^{k}$-torsor
is an affine surface. Indeed, writing $g=x^{-l}h\left(x\right)$ where
$h\in\mathbb{C}\left[x\right]\setminus x\mathbb{C}\left[x\right]$
and $l>\min\left(0,-k\right)$, the local regular functions 
\[
\varphi_{+}=x^{k+l}u_{+}+h\left(x\right)\in\Gamma(S\mid_{\bug_{+}},\mathcal{O}_{S})\qquad\textrm{and}\qquad\varphi_{-}=x^{l}u_{-}\in\Gamma(S\mid_{\bug_{-}},\mathcal{O}_{S})
\]
glue to a global one $\varphi\in\Gamma(S,\mathcal{O}_{S})$ for which
the morphism $\pi=\left(\delta\circ\rho,\varphi\right):S\rightarrow\mathbb{A}^{2}={\rm Spec}\left(\mathbb{C}\left[x,y\right]\right)$
maps the fibers $\rho^{-1}\left(o_{\pm}\right)$ to the distinct points
$\left(0,h\left(0\right)\right)$ and $\left(0,0\right)$ respectively
and restricts to an isomorphism $S\setminus\rho^{-1}\left(\left\{ o_{\pm}\right\} \right)\simeq{\rm Spec}\left(\mathbb{C}\left[x^{\pm1},y\right]\right)$.
Since the inverse images by $\pi$ of the principal affine open subsets
$y\neq h\left(0\right)$ and $y\neq0$ of $\mathbb{A}^{2}$ are principal
open subsets of $S\setminus\rho^{-1}\left(o_{+}\right)\simeq\mathbb{A}^{2}$
and $S\setminus\rho^{-1}\left(o_{-}\right)\simeq\mathbb{A}^{2}$ respectively,
it follows that $\pi:S\rightarrow\mathbb{A}^{2}$ is an affine morphism
whence that $S$ is an affine scheme. Note that conversely the total
space of a trivial $\mathbb{L}^{k}$-torsor cannot be affine since
it is not even separated. 

\end{parn}
\begin{example}
\label{exa:Sd_surfaces} For every $d\in\mathbb{Z}$, we let $\zeta_{d}:S\left(d\right)\rightarrow\bug$
be the nontrivial $\mathbb{L}^{d}$-torsor with gluing isomorphism
\[
\bug_{+}\times\mathbb{A}^{1}\supset\bug_{+}^{*}\times\mathbb{A}^{1}\stackrel{\sim}{\rightarrow}\bug_{-}^{*}\times\mathbb{A}^{1}\subset\bug_{-}\times\mathbb{A}^{1},\;\left(x,u_{+}\right)\mapsto\left(x,x^{d}u_{+}+x^{\min\left(-1,d-1\right)}\right).
\]
One checks easily that $\zeta_{-d}:S\left(-d\right)\rightarrow\bug$
is isomorphic to the pull-back $S\left(d\right)\times_{\bug}\bug$
of $\zeta_{d}:S\left(d\right)\rightarrow\bug$ by the automorphism
$\theta$ of $\bug$ which exchanges the open subsets $\bug_{\pm}$.
Since $\Omega_{\bug}^{1}$ is trivial it follows that $K_{S\left(d\right)}\simeq\Omega_{S(d)/\bug}^{1}\simeq\zeta_{d}^{*}\mathbb{L}^{-d}$
whence that ${\rm Pic}(S(d))/\langle K_{S(d)}\rangle\simeq\mathbb{Z}/\left|d\right|\mathbb{Z}$.
Therefore, the surfaces $S(d)$ are pairwise non isomorphic as schemes
over $\bug$ while $S(d)$ and $S(d')$ are isomorphic as abstract
schemes if and only if $d=\pm d'$. Note that since ${\rm Pic}\left(S(d)\right)\simeq{\rm Pic}(S(d)\times_{\bug}\mathbb{A}_{\bug}^{r})$
for every $r\geq1$, the same argument shows more generally that $S(d)\times_{\bug}\mathbb{A}_{\bug}^{r}$
is isomorphic to $S(d')\times_{\bug}\mathbb{A}_{\bug}^{r}$ as a scheme
over $\bug$ if and only if $d=d'$, and as an abstract scheme if
and only if $d=\pm d'$. 

If $d\geq0$ then, letting $\varphi\in\Gamma(S\left(d\right),\mathcal{O}_{S\left(d\right)})$
be defined locally by $(\varphi_{+},\varphi_{-})=(x^{d+1}u_{+}+1,xu_{-})$
as in \ref{Par:Rank-one-bug}, one checks that the rational functions
$\psi=x^{-1}\varphi\left(\varphi-1\right)$ and $\xi=x^{-d}\varphi^{d}\psi$
on $S\left(d\right)$ are regular and that the morphism $\left(\delta\circ\zeta_{d},\varphi,\psi,\xi\right):S\left(d\right)\rightarrow\mathbb{A}^{4}={\rm Spec}\left(\mathbb{C}\left[x,y,z,u\right]\right)$
is a closed embedding of $S\left(d\right)$ as the surface defined
by the equations 
\[
xz=y\left(y-1\right),\;\left(y-1\right)^{d}u=z^{d+1},\; x^{d}u=y^{d}z.
\]

\end{example}
The following result shows that for a non trivial affine-linear bundle
of rank one $\rho:S\rightarrow\bug$, the isomorphy type of $S$ as
an abstract scheme is essentially uniquely determined by its one as
an affine-linear bundle over $\bug$: 
\begin{thm}
Two non trivial affine-linear bundles or rank one $\rho_{i}:S_{i}\rightarrow\bug$,
$i=1,2$, have isomorphic total spaces if and only if their isomorphy
classes in $H^{1}(\bug,{\rm Aff}_{1})$ belong to the same orbit of
the action of ${\rm Aut}\left(\bug\right)$. \end{thm}
\begin{proof}
The condition is clearly sufficient. Conversely, if either $S_{1}$
or $S_{2}$ admits a unique affine-linear bundle structure over $\bug$
up to composition by automorphisms of $\bug$ then both admit a unique
such structure and so every isomorphism $\Phi:S_{2}\stackrel{\sim}{\rightarrow}S_{1}$
descends to an automorphism $\varphi$ of $\bug$ such that $\rho_{1}\circ\Phi=\varphi\circ\rho_{2}$.
This implies in turn that $\Phi$ factors through an isomorphism of
affine-linear bundles $\Phi':S_{2}\rightarrow S_{1}\times_{\bug}\bug$
whence that the isomorphy classes in $H^{1}(\bug,{\rm Aff}_{1})$
of the ${\rm Aff}_{1}$-bundles associated to $S_{1}$ and $S_{2}$
belong to a same orbit of ${\rm Aut}\left(\bug\right)$. Otherwise,
if $S_{1}$ and $S_{2}$ both admit at least two affine-linear bundle
structures over $\bug$ with distinct general fibers then, by combining
Theorem 3.11 and 5.3 in \cite{DubG03}, we obtain the following :
if the canonical bundles $K_{S_{1}}$ and $K_{S_{2}}$ are both trivial
then $\rho_{i}:S_{i}\rightarrow\bug$, $i=1,2$, are both isomorphic
to $\zeta_{0}:S\left(0\right)\rightarrow\bug$. Otherwise, for $d={\rm ord}({\rm Pic}(S_{1})/\langle K_{S_{1}}\rangle)={\rm ord}({\rm Pic}(S_{2})/\langle K_{S_{2}}\rangle)$,
$\rho_{i}:S_{i}\rightarrow\bug$ is isomorphic as an affine-linear
bundle either to the one $\zeta_{d}:S\left(d\right)\rightarrow\bug$
or to the one $\zeta_{-d}:S\left(-d\right)\rightarrow\bug$. This
completes the proof since the latters are obtained from each others
via the base change by the automorphism $\theta$ of $\bug$ (see
Example \ref{exa:Sd_surfaces} above). 
\end{proof}

\subsection{Affine-linear bundles of higher ranks}

\indent\newline\noindent In this subsection, we consider affine-linear
bundles $\rho:V\rightarrow\bug$ of rank $r\geq2$. Recall that to
every such bundle is associated a vector bundle $E\rightarrow\bug$
unique up to isomorphism for which $\rho:V\rightarrow\bug$ inherits
the structure of an $E$-torsor. In contrast with the case of affine-linear
bundles of rank one, we have the following characterization: 
\begin{thm}
\label{thm:AffB-on-bug} Let $p_{i}:E_{i}\rightarrow\bug$, $i=1,2$
be vector bundles of the same rank $r\geq2$ and let $\rho_{i}:V_{i}\rightarrow\bug$
be non trivial $E_{i}$-torsors, $i=1,2$. Then the following holds: 

1) $V_{1}$ and $V_{2}$ are isomorphic as schemes over $\bug$ if
and only if $\deg(E_{1})=\deg(E_{2})$,

2) $V_{1}$ and $V_{2}$ are isomorphic as abstract schemes if and
only if $\deg(E_{1})=\pm\deg(E_{2})$. 
\end{thm}
\begin{parn} Theorem \ref{thm:AffB-on-bug} is a consequence of Lemmas
\ref{lem:TorsB-on-bug} and Proposition \ref{prop:Aff-bug-K0} below
which, combined with Lemma \ref{lem:K0-bug}, imply that the total
space of non trivial $E$-torsor $\rho:V\rightarrow\bug$ or rank
$r\geq2$ is isomorphic as a scheme over $\bug$ to $S\left(d\right)\times_{\bug}\mathbb{A}_{\bug}^{r-1}$,
where $d=-\deg\left(E\right)$ and where $\zeta_{d}:S(d)\rightarrow\bug$
is the non trivial $\mathbb{L}^{d}$-torsor defined in Example \ref{exa:Sd_surfaces}
above. 

\end{parn}
\begin{lem}
\label{lem:TorsB-on-bug} The total spaces of all non trivial torsors
under a fixed vector bundle $p:E\rightarrow\bug$ of rank $r\geq2$
are affine and isomorphic as schemes over $\bug$. \end{lem}
\begin{proof}
By virtue of \ref{par:VectB-on-bug} above, we may assume that $E=\bigoplus_{i=1}^{r}\mathbb{L}^{k_{i}}$,
$k_{1},\ldots,k_{r}\in\mathbb{Z}$. Given a non trivial $E$-torsor
$\rho:V\rightarrow\bug$ there exists an index $i$ such that the
$i$-th component of the isomorphy class $(v_{1},\ldots,v_{r})$ of
$V$ in $H^{1}(\bug,E)\simeq\bigoplus_{i=1}^{r}H^{1}(\bug,\mathbb{L}^{k_{i}})$
is not zero. Up to a permutation, we may assume from now on that $v_{1}\neq0$.
The actions of $\mathbb{L}^{k_{1}}$ and $E_{1}=\bigoplus_{i=2}^{r}\mathbb{L}^{k_{i}}\simeq E/\mathbb{L}^{k_{1}}$
on $V$ commute and we have a cartesian square \[\xymatrix{ V \ar[r]^-{\pi_1} \ar[d] & S_1=V/E_1 \ar[d]^{\rho_1} \\ V/\mathbb{L}^{k_1} \ar[r] & \bug }\]where
$\rho_{1}:S_{1}\rightarrow\bug$ is an $\mathbb{L}^{k_{1}}$-torsor
with isomorphy class $v_{1}\in H^{1}(\bug,\mathbb{L}^{k_{1}})$ and
where $\pi_{1}:V\rightarrow S_{1}=V/E_{1}$ is a $\rho_{1}^{*}E_{1}$-torsor.
Since $\rho_{1}:S_{1}\rightarrow\bug$ is a non trivial torsor, $S_{1}$
is an affine scheme by virtue of \ref{Par:Rank-one-bug} and so $\pi_{1}:V\rightarrow S_{1}$
is isomorphic as a scheme over $\bug$ to the total space of the trivial
$\rho_{1}^{*}E_{1}$-torsor ${\rm p}_{1}:S_{1}\times_{\bug}E_{1}\rightarrow S_{1}$.
In particular, $V$ is an affine scheme. 

With the notation of Example \ref{exa:Sd_surfaces} above, we claim
that $S_{1}\times_{\bug}E_{1}$ is isomorphic as a scheme over $\bug$
to the $r$-fold fiber product $S(k_{1})\times_{\bug}\cdots\times_{\bug}S(k_{r})$.
Indeed, since for every $k\in\mathbb{Z}$, $\zeta_{k}:S\left(k\right)\rightarrow\bug$
is an $\mathbb{L}^{k}$-torsor, the fiber product $S_{1}\times_{\bug}S\left(k\right)$
is simultaneously the total space of a $\rho_{1}^{*}\mathbb{L}^{k}$-torsor
over $S_{1}$ and of a $\zeta_{k}^{*}\mathbb{L}^{k_{1}}$-torsor over
$S\left(k\right)$ via the first and the second projection respectively.
The fact that $S_{1}$ and $S\left(k\right)$ are both affine implies
that the latter are both trivial torsors, which yields isomorphisms
$S_{1}\times_{\bug}\mathbb{L}^{k}\simeq S_{1}\times_{\bug}S\left(k\right)\simeq\mathbb{L}^{k_{1}}\times_{\bug}S\left(k\right)$
of schemes over $\bug$. The same argument applied to the non trivial
$\mathbb{L}^{k_{1}}$-torsor $\zeta_{k_{1}}:S(k_{1})\rightarrow\bug$
provides isomorphisms $S(k_{1})\times_{\bug}\mathbb{L}^{k}\simeq S(k_{1})\times_{\bug}S\left(k\right)\simeq\mathbb{L}^{k_{1}}\times_{\bug}S\left(k\right)$
of schemes over $\bug$. Letting $E_{2}=E_{1}/\mathbb{L}^{k_{2}}\simeq\bigoplus_{i=3}^{r}\mathbb{L}^{k_{i}}$,
we finally obtain isomorphisms 
\[
S_{1}\times_{\bug}E_{1}\simeq S(k_{1})\times_{\bug}S(k_{2})\times_{\bug}E_{2}\simeq S(k_{1})\times_{\bug}S(k_{2})\times_{\bug}(S(k_{3})\times_{\bug}\cdots\times_{\bug}S(k_{r}))
\]
where the last isomorphism follows from the fact that the fiber product
of the affine scheme $q:S(k_{1})\times_{\bug}S(k_{2})\rightarrow\bug$
with the $E_{2}$-torsor $S(k_{3})\times_{\bug}\cdots\times_{\bug}S(k_{r})$
is isomorphic to the trivial $q^{*}E_{2}$-torsor $S(k_{1})\times_{\bug}S(k_{2})\times_{\bug}E_{2}$
over $S(k_{1})\times_{\bug}S(k_{2})$. \end{proof}
\begin{prop}
\label{prop:Aff-bug-K0} The isomorphy type of the total space of
a non trivial affine-linear bundle $\rho:V\rightarrow\bug$ or rank
$r\geq2$ as a scheme over $\bug$ depends only on the class in $K_{0}\left(\bug\right)$
of its associated vector bundle. \end{prop}
\begin{proof}
Given a vector bundle $E\rightarrow\bug$ of rank $r\geq2$, we will
show more precisely that the total space of a non trivial $E$-torsor
$\rho:V\rightarrow\bug$ is isomorphic as a scheme over $\bug$ to
$S\left(d\right)\times_{\bug}\mathbb{A}_{\bug}^{r-1}$, where $d=-\deg\left(E\right)$.
We proceed by induction on the rank of $E$. By combining \ref{par:VectB-on-bug}
and Lemma \ref{lem:TorsB-on-bug} above, we may assume that $E=\bigoplus_{i=1}^{r}\mathbb{L}^{k_{i}}$,
where $k_{1},\ldots,k_{r}\in\mathbb{Z}$ and $-k_{1}-\cdots-k_{r}=-d$
and that $V=S(k_{1})\times_{\bug}\cdots\times_{\bug}S(k_{r})$. Furthermore,
since the induction hypothesis implies that $S(k_{2})\times_{\bug}\cdots\times_{\bug}S(k_{r})\simeq S(d-k_{1})\times_{\bug}\mathbb{A}_{\bug}^{r-1}$
as schemes over $\bug$, it is enough to show that for every $m,n\in\mathbb{Z}$,
$S(m)\times_{\bug}S(n)$ and $S(m+n)\times_{\bug}\mathbb{A}_{\bug}^{1}$
are isomorphic as schemes over $\bug$. If $m$ or $n$ is equal to
zero then we are done. Otherwise, up to taking the pull-back of $S(m)\times_{\bug}S(n)$
by the automorphism $\theta$ of $\bug$ and exchanging the roles
of $m$ and $n$, we may assume similarly as in the proof of Lemma
\ref{lem:K0-bug} above that either $0<m\leq n$ or $m<0<n$. In the
first case, letting $E\simeq\mathbb{L}^{m+n}\oplus\mathbb{A}_{\bug}^{1}$
be the vector bundle on $\bug$ defined by the matrix 
\[
M=\left(\begin{array}{cc}
x^{m} & 1\\
0 & x^{n}
\end{array}\right)\in{\rm GL}_{2}\left(\mathbb{C}\left[x^{\pm1}\right]\right),
\]
it follows from Lemma \ref{lem:TorsB-on-bug} that the total space
of the non trivial $E$-torsor $\rho:V\rightarrow\bug$ with gluing
isomorphism 

\begin{eqnarray*}
\bug_{+}\times\mathbb{A}^{2}\supset\bug_{+}^{*}\times\mathbb{A}^{2} & \stackrel{\sim}{\rightarrow} & \bug_{-}^{*}\times\mathbb{A}^{2}\subset\bug_{-}\times\mathbb{A}^{2},\\
(x,(v_{+},u_{+})) & \mapsto & (x,(x^{m}v_{+}+u_{+},x^{n}u_{+}+x^{-1}))
\end{eqnarray*}
is isomorphic to $S(m+n)\times_{\bug}\mathbb{A}_{\bug}^{1}$. On the
other hand, since $E$ is an extension of $\mathbb{L}^{n}$ by $\mathbb{L}^{m}$,
$V$ inherits a free action of $\mathbb{L}^{m}$ whose quotient $V/\mathbb{L}^{m}$
coincides with the total space of the $\mathbb{L}^{n}$-torsor $\zeta_{n}:S(n)\rightarrow\bug$
with gluing isomorphism $(x,u_{+})\mapsto(x,x^{n}u_{+}+x^{-1})$.
Furthermore, the quotient morphism $V\rightarrow V/\mathbb{L}^{m}\simeq S(n)$
inherits the structure of a $\zeta_{n}^{*}\mathbb{L}^{m}$-torsor
whence is isomorphic to the trivial one $S(n)\times_{\bug}\mathbb{L}^{m}$
as $S(n)$ is affine. Summing up, we obtain isomorphisms $S(m+n)\times_{\bug}\mathbb{A}_{\bug}^{1}\simeq V\simeq S(n)\times_{\bug}\mathbb{L}^{m}\simeq S(n)\times_{\bug}S\left(m\right)$
of schemes over $\bug$. 

The case $m<0<n$ follows from a similar argument starting from the
vector bundle $E\simeq\mathbb{L}^{m}\oplus\mathbb{L}^{n}$ defined
by the matrix 
\[
N=\left(\begin{array}{cc}
1 & x^{m}\\
0 & x^{m+n}
\end{array}\right)\in{\rm GL}_{2}\left(\mathbb{C}\left[x^{\pm1}\right]\right),
\]
and the non trivial $E$-torsor $\rho:V\rightarrow\bug$ with gluing
isomorphism 
\begin{eqnarray*}
\bug_{+}\times\mathbb{A}^{2}\supset\bug_{+}^{*}\times\mathbb{A}^{2} & \stackrel{\sim}{\rightarrow} & \bug_{-}^{*}\times\mathbb{A}^{2}\subset\bug_{-}\times\mathbb{A}^{2}\\
(x,(v_{+},u_{+})) & \mapsto & (x,(v_{+}+x^{m}u_{+},x^{m+n}u_{+}+x^{\min\left(-1,m+n-1\right)})).
\end{eqnarray*}
 
\end{proof}

\section{Isomorphy types of complements of hyperplane sub-bundles }

In this section, we consider total spaces of non trivial affine-linear
bundles $\nu:V\rightarrow\mathbb{P}^{1}$ over the projective line.
We first review the case of affine-linear bundles of rank one: the
crucial observation there is the fact that the total space of a non
trivial $\mathcal{O}_{\mathbb{P}^{1}}(-d)$-torsor $\nu:V\rightarrow\mathbb{P}^{1}$,
where $d\geq2$, is isomorphic to the affine surface $\zeta_{d-2}:S\left(d-2\right)\rightarrow\bug$
of Example \ref{exa:Sd_surfaces}, whence admits the structure of
a non trivial $\mathbb{L}^{d-2}$-torsor over the affine line with
a double origin. This enables to consider total spaces of non trivial
affine-linear bundles $\nu:V\rightarrow\mathbb{P}^{1}$ of higher
ranks as being simultaneously that of certain non trivial affine-linear
bundles over $\bug$, and to deduce the classification of total spaces
of such bundles as a particular case of the results established in
the first section.

\subsection{Affine linear bundles of rank one and the Danilov-Gizatullin Theorem}

\indent\newline\noindent The Danilov-Gizatullin Theorem \cite[Theorem 5.8.1 ]{Gizatullin1977}
asserts that the isomorphy type of the complement of an ample section
$C$ in a Hirzebruch surface $\pi_{n}:\mathbb{F}_{n}=\mathbb{P}\left(\mathcal{O}_{\mathbb{P}^{1}}\oplus\mathcal{O}_{\mathbb{P}^{1}}\left(n\right)\right)\rightarrow\mathbb{P}^{1}$,
$n\geq0$, depends only on the self-intersection $(C^{2})\geq2$ of
$C$. Since there is a one-to-one correspondence between non trivial
$\mathcal{O}_{\mathbb{P}^{1}}\left(-d\right)$-torsors $\nu:V\rightarrow\mathbb{P}^{1}$
and complements of ample sections $C$ with self-intersection $(C^{2})=d$
in Hirzebruch surfaces \cite[Remark 4.8.6]{Gizatullin1977}, the Danilov-Gizatullin
Theorem can be rephrased as the fact that the isomorphy type of the
total space of a non trivial $\mathcal{O}_{\mathbb{P}^{1}}(-d)$-torsor
$\nu:V\rightarrow\mathbb{P}^{1}$ depends only on $d$ and not on
its isomorphy class as a torsor in $H^{1}(\mathbb{P}^{1},\mathcal{O}_{\mathbb{P}^{1}}(-d))$.
We have the following more effective result: 
\begin{prop}
\label{prop:Explicit-GizDan} The total space of a non trivial $\mathcal{O}_{\mathbb{P}^{1}}\left(-d\right)$-torsor
$\nu:V\rightarrow\mathbb{P}^{1}$, $d\geq2$, is isomorphic to the
surface $\zeta_{d-2}:S\left(d-2\right)\rightarrow\bug$ of Example
\ref{exa:Sd_surfaces}. Furthermore the isomorphism $V\simeq S(d-2)$
can be chosen in such a way that $\nu^{*}\mathcal{O}_{\mathbb{P}^{1}}\left(1\right)=\zeta_{d-2}^{*}\mathbb{L}^{-1}$
in ${\rm Pic}\left(V\right)\simeq\mathbb{Z}$. \end{prop}
\begin{proof}
Recall that for every $d\geq2$, $\zeta_{d-2}:S(d-2)\rightarrow\bug$
is isomorphic to the surface in $\mathbb{A}^{4}={\rm Spec}(\mathbb{C}\left[x,y,z,u\right])$
defined by the equations 
\[
xz=y\left(y-1\right),\;\left(y-1\right)^{d-2}u=z^{d-1},\; x^{d-2}u=y^{d-2}z.
\]
Letting $\mathbb{P}^{1}={\rm Proj}(\mathbb{C}[w_{0},w_{1}])$, $U_{0}=\mathbb{P}^{1}\setminus\{\left[1:0\right]\}={\rm Spec}\left(\mathbb{C}\left[w\right]\right)$
and $U_{\infty}=\mathbb{P}^{1}\setminus\{\left[0:1\right]\}={\rm Spec}\left(\mathbb{C}\left[w'\right]\right)$
where $w=w_{0}/w_{1}$ and $w'=w_{1}/w_{0}$, one checks that the
morphism $\nu_{d-2}:S(d-2)\rightarrow\mathbb{P}^{1}$, $(x,y,z,u)\mapsto\left[x:y\right]=\left[y-1:z\right]$
defines an $\mathcal{O}_{\mathbb{P}^{1}}(-d)$-torsor with local trivializations
\[
\tau_{0}:\nu_{d-2}^{-1}(U_{0})\stackrel{\sim}{\rightarrow}{\rm Spec}\left(\mathbb{C}\left[w\right]\left[u\right]\right),\quad\tau_{\infty}:\nu_{d-2}^{-1}(U_{\infty})\stackrel{\sim}{\rightarrow}{\rm Spec}\left(\mathbb{C}\left[w'\right]\left[x\right]\right)
\]
and transition isomorphism $\tau_{\infty}\circ\tau_{0}^{-1}\mid_{U_{0}\cap U_{\infty}}$
given by $\left(w,u\right)\mapsto\left(w',x\right)=\left(w^{-1},w^{d}u+w\right)$.
Furthermore, it follows from the construction of $\zeta_{d-2}:S\left(d-2\right)\rightarrow\bug$
and $\nu_{d-2}:S(d-2)\rightarrow\mathbb{P}^{1}$ that $\nu_{d-2}^{-1}(\left[0:1\right])=\zeta_{d-2}^{-1}(o_{+})$.
Since the classes of these divisors in ${\rm Cl}\left(S(d-2)\right)\simeq{\rm Pic}\left(S(d-2)\right)$
coincide respectively with the line bundles $\nu_{d-2}^{*}\mathcal{O}_{\mathbb{P}^{1}}(1)$
and $\zeta_{d-2}^{*}\mathbb{L}^{-1}$, the assertion follows from
the ``refined'' Danilov-Gizatullin Theorem \cite[Theorem 3.1]{Dubouloz2011a}
which asserts that if $\nu_{i}:V_{i}\rightarrow\mathbb{P}^{1}$, $i=1,2$,
are non trivial $\mathcal{O}_{\mathbb{P}^{1}}\left(-d\right)$-torsors,
then there exists an isomorphism $f:V_{1}\stackrel{\sim}{\rightarrow}V_{2}$
such that $f^{*}(\nu_{2}^{*}\mathcal{O}_{\mathbb{P}^{1}}\left(1\right))\simeq\nu_{1}^{*}\mathcal{O}_{\mathbb{P}^{1}}\left(1\right)$. 
\end{proof}

\subsection{Affine-linear bundles of arbitrary ranks}

\indent\newline\noindent By combining our previous results, we obtain
the following characterization: 
\begin{thm}
\label{thm:MainThmP1} Let $p_{i}:E_{i}\rightarrow\mathbb{P}^{1}$,
$i=1,2$ be vector bundles of the same rank $r\geq1$ and let $\nu_{i}:V_{i}\rightarrow\mathbb{P}^{1}$
be non trivial $E_{i}$-torsors, $i=1,2$. Then $V_{1}$ and $V_{2}$
are affine, and isomorphic as abstract varieties if and only if $\deg(\det E_{1}^{\vee}\otimes\Omega_{\mathbb{P}^{1}}^{1})=\pm\deg(\det E_{2}^{\vee}\otimes\Omega_{\mathbb{P}^{1}}^{1})$. \end{thm}
\begin{proof}
The argument is very similar to the one used in the proof of Lemma
\ref{lem:TorsB-on-bug} above. Recall that every vector bundle $E\rightarrow\mathbb{P}^{1}$
of rank $r\geq2$ splits into a direct sum $E=\bigoplus_{i=1}^{r}\mathcal{O}_{\mathbb{P}^{1}}(k_{i})$,
$k_{1},\ldots,k_{r}\in\mathbb{Z}$, of line bundles \cite{Grothendieck1957}.
Therefore, if $\nu:V\rightarrow\mathbb{P}^{1}$ is a non trivial $E$-torsor,
then there exists an index $i\in\{1,\ldots,r\}$ such that the $i$-th
component of its isomorphy class $(a_{1},\ldots,a_{r})\in H^{1}(\mathbb{P}^{1},E)\simeq\bigoplus_{i=1}^{r}H^{1}(\mathbb{P}^{1},\mathcal{O}_{\mathbb{P}^{1}}(k_{i}))$
is non zero. Letting $E_{i}=\bigoplus_{j\neq i}\mathcal{O}_{\mathbb{P}^{1}}(k_{j})\simeq E/\mathcal{O}_{\mathbb{P}^{1}}(k_{i})$,
the quotient of $V$ by the induced action of $E_{i}$ inherits the
structure of a non trivial $\mathcal{O}_{\mathbb{P}^{1}}(k_{i})$-torsor
$\nu_{i}:S_{i}\rightarrow\mathbb{P}^{1}$ with isomorphy class $a_{i}\in H^{1}(\mathbb{P}^{1},\mathcal{O}_{\mathbb{P}^{1}}(k_{i}))$.
Furthermore, the quotient morphism $V\rightarrow S_{i}=V/E_{i}$ has
the structure of a $\nu_{i}^{*}E_{i}$-torsor. Proposition \ref{prop:Explicit-GizDan}
above implies that $k_{i}=-d$ for some $d\geq2$ and that $S_{i}$
is isomorphic to the surface $\zeta_{d-2}:S(d-2)\rightarrow\bug$.
In particular, $S_{i}\simeq S(d-2)$ is affine and so $V$ is isomorphic
to the trivial $\nu_{i}^{*}E_{i}$-torsor $S_{i}\times_{\mathbb{P}^{1}}E_{i}$.
Moreover, by choosing the isomorphism $S_{i}\simeq S(d-2)$ in such
a way that $\nu_{i}^{*}\mathcal{O}_{\mathbb{P}^{1}}(1)\simeq\zeta_{d-2}^{*}\mathbb{L}^{-1}$,
we obtain that $S_{i}\times_{\mathbb{P}^{1}}E_{i}$ is an affine variety,
isomorphic as a scheme over $\bug$ to the total space of a non trivial
torsor under the vector bundle $\tilde{E}=\mathbb{L}^{d-2}\oplus\mathbb{L}^{-k_{2}}\oplus\cdots\oplus\mathbb{L}^{-k_{r}}$.
Since $\deg\tilde{E}=-\deg(\det E^{\vee}\otimes\Omega_{\mathbb{P}^{1}}^{1})$,
the assertion follows from Theorem \ref{thm:AffB-on-bug} above. \end{proof}
\begin{example}
Let us consider again the example given in the introduction of an
affine variety which is simultaneously the total space of non trivial
torsors under complex vector bundles of different topological types.
The Euler exact sequence $0\rightarrow\Omega_{\mathbb{P}^{1}}^{1}\stackrel{j}{\rightarrow}\mathcal{O}_{\mathbb{P}^{1}}(-1)^{\oplus2}\rightarrow\mathcal{O}_{\mathbb{P}^{1}}\rightarrow0$
on $\mathbb{P}^{1}$ defines a non trivial $\Omega_{\mathbb{P}^{1}}^{1}$-torsor
$v:V\rightarrow\mathbb{P}^{1}$ with total space isomorphic to the
complement  of the diagonal $\Delta\simeq\mathbb{P}(\Omega_{\mathbb{P}^{1}}^{1})$
in $\mathbb{P}(\mathcal{O}_{\mathbb{P}^{1}}(-1)\oplus\mathcal{O}_{\mathbb{P}^{1}}(-1))\simeq\mathbb{P}^{1}\times\mathbb{P}^{1}$.
For every $k\in\mathbb{Z}$, the variety $V_{k}=V\times_{\mathbb{P}^{1}}\mathcal{O}_{\mathbb{P}^{1}}\left(k\right)\rightarrow\mathbb{P}^{1}$
is then a torsor under the vector bundle $F_{k}=\Omega_{\mathbb{P}^{1}}^{1}\oplus\mathcal{O}_{\mathbb{P}^{1}}\left(k\right)$.
Since $\deg(\det(F_{-k}^{\vee})\otimes\Omega_{\mathbb{P}^{1}}^{1})=k=-\deg(\det(F_{k}^{\vee})\otimes\Omega_{\mathbb{P}^{1}}^{1})$,
Theorem \ref{thm:MainThmP1} implies that $V_{-k}$ and $V_{k}$ are
isomorphic affine varieties. The exact sequence 
\[
0\rightarrow F_{k}=\Omega_{\mathbb{P}^{1}}^{1}\oplus\mathcal{O}_{\mathbb{P}^{1}}(k)\stackrel{j\oplus{\rm id}}{\rightarrow}E_{k}=\mathcal{O}_{\mathbb{P}^{1}}(-1)^{\oplus2}\oplus\mathcal{O}_{\mathbb{P}^{1}}(k)\rightarrow\mathcal{O}_{\mathbb{P}^{1}}\rightarrow0
\]
provides an open embedding of $V_{k}\simeq V_{-k}$ into $\mathbb{P}(E_{k})$
as the complement of the hyperplane sub-bundle $H_{k}=\mathbb{P}(F_{k})$.
Note that if $k=0$ then $H_{0}$ is nef but not ample and that if
$k>2$ then $H_{-k}$ is ample whereas $H_{k}$ has negative self-intersection
$(H_{k}^{3})=2-k$. 

More generally, for any $n\geq2$, the complement $V$ in $\mathbb{P}^{n}\times\mathbb{P}^{n}={\rm Proj}(\mathbb{C}[x_{0},\ldots,x_{n}])\times{\rm Proj}(\mathbb{C}[y_{0},\ldots,y_{n}])$
of the ample divisor $D$ with equation $\sum_{i=0}^{n}x_{i}y_{i}=0$
inherits simultaneously via the first and the second projection the
structure of an $\Omega_{\mathbb{P}^{n}}^{1}$-torsor $\nu_{i}:V\rightarrow\mathbb{P}^{n}$
associated with the Euler exact sequence $0\rightarrow\Omega_{\mathbb{P}^{n}}^{1}\rightarrow\mathcal{O}_{\mathbb{P}^{n}}(-1)^{\oplus n+1}\rightarrow\mathcal{O}_{\mathbb{P}^{n}}\rightarrow0$
on each factor in $\mathbb{P}^{n}\times\mathbb{P}^{n}$. Since $D$
is of type $(1,1)$ in ${\rm Pic}\left(\mathbb{P}^{n}\times\mathbb{P}^{n}\right)\simeq{\rm p}_{1}^{*}{\rm Pic}(\mathbb{P}^{n})\oplus{\rm p}_{2}^{*}{\rm Pic}\left(\mathbb{P}^{n}\right)$,
we have for every $k\in\mathbb{Z}$, $\nu_{2}^{*}\mathcal{O}_{\mathbb{P}^{n}}(k)=\nu_{1}^{*}\mathcal{O}_{\mathbb{P}^{n}}(-k)$
in ${\rm Pic}\left(V\right)\simeq\mathbb{Z}$. Therefore, similarly
as in the previous case, we may interpret $V\times_{\nu_{1},\mathbb{P}^{n}}\mathcal{O}_{\mathbb{P}^{n}}(-k)\simeq V\times_{\nu_{2},\mathbb{P}^{n}}\mathcal{O}_{\mathbb{P}^{n}}(k)$
as being simultaneously the total space of an $\Omega_{\mathbb{P}^{n}}^{1}\oplus\mathcal{O}_{\mathbb{P}^{n}}(-k)$-torsor
and of an $\Omega_{\mathbb{P}^{n}}^{1}\oplus\mathcal{O}_{\mathbb{P}^{n}}(k)$-torsor
over $\mathbb{P}^{n}$ via the first and the second projection respectively. 
\end{example}
\bibliographystyle{amsplain}

\end{document}